\documentclass[11pt]{article}

\usepackage{amsmath}
\usepackage{graphicx}
\usepackage{epstopdf}
\usepackage{amssymb}
\usepackage{mathrsfs}
\usepackage{amsfonts}
\usepackage{mathrsfs,amscd,amssymb,amsthm,amsmath,bm,graphicx,psfrag}
\usepackage[numbers,sort&compress]{natbib}
\pdfoutput=1

\makeatletter
\renewcommand{\@seccntformat}[1]{{\csname the#1\endcsname}{\normalsize.}\hspace{.5em}}

\makeatother

\numberwithin{equation}{section}

\def \[{\begin{equation}}
\def \]{\end{equation}}

\newtheorem{thm}{Theorem}[section]

\newtheorem{lem}[thm]{Lemma}

\newtheorem{con}[thm]{Conjecture}
\newtheorem{rem}[thm]{Remark}

\newenvironment{kst}
{\setlength{\leftmargini}{1.3\parindent}
 \begin{itemize}
 \setlength{\itemsep}{-1.1mm}}
{\end{itemize}}

\begin{document}
\setlength{\baselineskip}{16.5pt}

\begin{center}{\Large \bf A stability result for Berge-$K_{3,t}$ $r$-graphs and its applications}

\vspace{4mm} {\large Junpeng Zhou$^{\rm{a,b}}$, Xiying Yuan$^{\rm{a,b}*}$, Wen-Huan Wang$^{\rm{a,b}}$} \vspace{2mm}

{\small $^{\rm{a}}$ Department of Mathematics, Shanghai University, Shanghai 200444, P.R. China}

{\small $^{\rm{b}}$ Newtouch Center for Mathematics of Shanghai University, Shanghai 200444, P.R. China}

\end{center}

\footnotetext{*Corresponding author. Email address: xiyingyuan@shu.edu.cn (Xiying Yuan)}

\footnotetext{junpengzhou@shu.edu.cn (Junpeng Zhou)}

\footnotetext{whwang@shu.edu.cn (Wen-Huan Wang)}

\footnotetext{This work is supported by the National Natural Science Foundation of China (Nos.11871040, 12271337, 12371347)}

\begin{abstract}
An $r$-uniform hypergraph ($r$-graph) is linear if any two edges intersect at most one vertex. For a graph $F$, a hypergraph $H$ is Berge-$F$ if there is a bijection $\phi:E(F)\rightarrow E(H)$ such that $e\subseteq \phi(e)$ for all $e$ in $E(F)$.
In this paper, a kind of stability result for Berge-$K_{3,t}$ linear $r$-graphs is established. Based on this stability result, an upper bound for the linear Tur\'{a}n number of Berge-$K_{3,t}$ is determined.
For an $r$-graph $H$, let $\mathcal{A}(H)$ be the adjacency tensor of $H$. The spectral radius of $H$ is the spectral radius of the tensor $\mathcal{A}(H)$.
Some bounds for the maximum spectral radius of connected Berge-$K_{3,t}$-free linear $r$-graphs are obtained. 
\end{abstract}


\section{\normalsize Introduction}
\ \ \ \
A hypergraph $H=(V(H),E(H))$ consists of a vertex set $V(H)$ and an edge set $E(H)$, where each entry in $E(H)$ is a nonempty subset of $V(H)$. If $|e|=r$ for any $e\in E(H)$, then $H$ is called an $r$-graph ($r$-uniform hypergraph). A hypergraph $H$ is linear if any two edges intersect at most one vertex. For $v\in V(H)$, let $E_v(H)=\{e\in E(H)\ |\ v\in e\}$. The degree $d_v(H)$ of vertex $v$ (or simply $d_v$) is defined as $|E_v(H)|$. 
The neighborhood $N_v(H)$ of vertex $v$ is the set of all vertices adjacent to $v$ in $H$.

Let $F=(V(F),E(F))$ be a simple graph. A hypergraph $H$ is called Berge-$F$ if there is a bijection $\phi: E(F)\rightarrow E(H)$ such that $e\subseteq \phi(e)$ for all $e$ in $E(F)$. The graph $F$ is called a skeleton of Berge-$F$ hypergraphs \cite{D4}. 
Let $\mathcal{F}$ be a given family of hypergraphs. A hypergraph $H$ is called $\mathcal{F}$-free if $H$ does not contain any hypergraphs in $\mathcal{F}$. The Tur\'{a}n number ${\rm{ex}}_r(n,\mathcal{F})$ (linear Tur\'{a}n number ${\rm{ex}}^{\rm{lin}}_r(n,\mathcal{F})$) of $\mathcal{F}$ is the maximum number of edges of all $\mathcal{F}$-free $r$-graph ($\mathcal{F}$-free linear $r$-graph) with $n$ vertices.

Tur\'{a}n problems in graphs and hypergraphs are the central topics of extremal combinatorics \cite{B1}. In 1907, Mantel \cite{A222} proved Mantel theorem that ${\rm{ex}}_2(n,K_{3})\leq \lfloor\frac{n^2}{4}\rfloor$. In \cite{A333}, Tur\'{a}n determined ${\rm{ex}}_2(n,K_{t+1})$ in 1941.
For Berge hypergraphs, Gy\H{o}ri and Lemons \cite{B2} studied Tur\'{a}n numbers of Berge-$C_{2k}$ and Berge-$C_{2k+1}$. Gerbner and Palmer \cite{B3} gave some bounds for the Tur\'{a}n number of Berge-$K_{s,t}$ for $r$-graphs, where $r\geq s+t$ and $K_{s,t}$ is a complete bipartite graph with two parts of order $s$ and $t$. Gerbner et al. \cite{B5} determined the asymptotics for the Tur\'{a}n number of Berge-$K_{2,t}$ for $3$-graphs. They also proved an upper bound for the linear Tur\'{a}n number of Berge-$K_{2,t}$.
Kang et al. \cite{B7} studied the Tur\'{a}n number of Berge-$F$ when $F$ is a matching of size $k+1$.
Khormali and Palmer \cite{D4} proved an asymptotic result for the Tur\'{a}n number of Berge star forests.
Ghosh et al. \cite{C333} proved an upper bound for the Tur\'{a}n number of Berge-$k$-book for $3$-graphs.
For more results on the (linear) Tur\'{a}n problems of Berge hypergraphs, one can refer to \cite{B2,B3,B4,A666,B5,C5,C4,B7,D4,C111,C222,C333,C444,E333}.

The purpose of this paper is to study the linear Tur\'{a}n problem of Berge-$K_{3,t}$. A kind of stability result for Berge-$K_{3,t}$ linear $r$-graphs is established. Based on this stability result, we determine an upper bound for the linear Tur\'{a}n number of Berge-$K_{3,t}$.

\begin{thm} Suppose that $n\geq r\geq3$ and $t\geq3$. Then
\begin{eqnarray*}
{\rm{ex}}_r^{\rm{lin}}(n,{\text{{\rm{Berge}}-}}K_{3,t})\leq \frac{(r(r-3)+4)n^2}{2r(r-1)^2}+O(n^\frac{3}{2}).
\end{eqnarray*}
\end{thm}

A tensor $\mathcal{T}$ with order $r$ and dimension $n$ over the complex field $\mathbb{C}$ is a multi-array $\mathcal{T}=(\mathcal{T}_{i_1i_2\cdots i_r})$, where $i_j\in [n]:=\{1,2,\cdots,n\}$ and $j\in [r]$. A tensor $\mathcal{T}$ is called symmetric if $\mathcal{T}_{i_1i_2\cdots i_r}=\mathcal{T}_{\sigma({i_1})\cdots\sigma({i_r})}$, where $\sigma$ is any permutation of its indices. For a vector ${\bf{x}}=(x_1,x_2,\cdots,x_n)^{\rm{T}}\in \mathbb{C}^n$, by the definition of product of tensors \cite{A1}, vector $\mathcal{T}{\bf{x}}^{r-1}$ is as follows
\begin{eqnarray}
(\mathcal{T}{\bf{x}}^{r-1})_i=\sum_{i_2,\cdots,i_r\in [n]}\mathcal{T}_{ii_2\cdots i_r}x_{i_2}\cdots x_{i_r}, \ i\in [n].
\end{eqnarray}
If there exists a number $\lambda \in \mathbb{C}$ and a nonzero vector ${\bf{x}}\in \mathbb{C}^n$ satisfying $$\mathcal{T}{\bf{x}}^{r-1}=\lambda {\bf{x}}^{[r-1]},$$ where ${\bf{x}}^{[r-1]}:=(x_1^{r-1},x_2^{r-1},\cdots,x_n^{r-1})^{\rm{T}}$, then $\lambda$ is called an eigenvalue of $\mathcal{T}$ and ${\bf{x}}$ is an eigenvector corresponding to $\lambda$ \cite{A4}.
The spectral radius of $\mathcal{T}$ is defined as $$\rho(\mathcal{T})=\max\{|\lambda|\,\big|\, \lambda \ \text{is an eigenvalue of} \ \mathcal{T}\}.$$

In 2012, Cooper and Dutle \cite{A4} gave the definition of adjacency tensor $\mathcal{A}(H)$ for an $r$-graph $H$. The adjacency tensor $\mathcal{A}(H)=(\mathcal{A}_{i_1i_2\cdots i_r})$ of an $r$-graph $H$ is an order $r$ dimension $n$ symmetric tensor, where
\begin{eqnarray*}
\mathcal{A}_{i_1i_2\cdots i_r}=
\begin{cases}
\frac{1}{(r-1)!}, & \text{if} \ \{i_1,i_2,\cdots,i_r\}\in E(H); \\
0, & \text{otherwise}.
\end{cases}
\end{eqnarray*}

Note that if an $r$-graph $H$ is connected, then $\mathcal{A}(H)$ is weakly irreducible \cite{A5}. By the Perron-Frobenius theorem for nonnegative tensors \cite{A3,A5,A6,A7}, we know that $\rho(\mathcal{A}(H))$ is an eigenvalue of $\mathcal{A}(H)$, and there is a unique positive eigenvector ${\bf{x}}=(x_1,x_2,\cdots,x_n)^{\rm{T}}$ with $\|{\bf{x}}\|_r=1$ corresponding to $\rho(\mathcal{A}(H))$. We call ${\bf{x}}$ the Perron vector of $H$. Throughout this paper, the spectral radius of an $r$-graph $H$, denoted by $\rho(H)$, is defined as the spectral radius of the adjacency tensor $\mathcal{A}(H)$.

The spectral hypergraph theory has become a hot topic in algebraic graph theory \cite{A4,E111,E222}.
The (linear) spectral Tur\'{a}n problem is a spectral version of the graph or hypergraph Tur\'{a}n problem, which aims to determine the maximum spectral radius of all $\mathcal{F}$-free graphs or (linear) hypergraphs with $n$ vertices.
Recently, the spectral Tur\'{a}n problem of hypergraphs has attracted much attention. In \cite{C4}, Hou et al. determined the maximum spectral radius of Fan$^r$-free linear $r$-graphs and gave an upper bound for the spectral radius of Berge-$C_4$-free linear $r$-graphs. Gao et al. \cite{C5} determined the maximum spectral radius of linear $r$-graphs without $r$-expansion of $K_{r+1}$ and characterized the extremal linear $r$-graph. Wang and Yu \cite{E1} gave an upper bound for the $\alpha$-spectral radius of Berge-$C_l$-free linear $3$-graphs.

Let ${\rm{spex}}_r^{\rm{lin}}(n,{\rm{\text{Berge-}}}K_{3,t})$ denote the maximum spectral radius of all connected Berge-$K_{3,t}$-free linear $r$-graphs with $n$ vertices.
For $n\geq r\geq2$, $t\geq3$ and $x,y>0$, set
\begin{eqnarray*}
f(n,r,t,x,y)&:=& \bigg(t(r-1)+\frac{(r-2)(n-r)}{2(r-1)}+\frac{n-r}{(r-1)^2}\bigg)x+\frac{1}{r-1}y \notag\\
&{}&-\bigg(t+r-2+\frac{(r-2)(n-r-2)}{2(r-1)}-\frac{(r-2)(r-3)(n-r)}{2(r-1)^2}\bigg).
\end{eqnarray*}

Based on this stability result that has been established, we also obtain an upper bound and a lower bound for the maximum spectral radius of connected Berge-$K_{3,t}$-free linear $r$-graphs.

\begin{thm} Suppose that $n\geq r\geq3$ and $t\geq3$. Then
\begin{eqnarray*}
{\rm{spex}}_r^{\rm{lin}}(n,{\text{{\rm{Berge}}-}}K_{3,t})\leq f^{\frac{1}{2}}\bigg(n,r,t,\frac{n-1}{r-1},\frac{n-1}{r-1}\bigg).
\end{eqnarray*}
\end{thm}

\begin{thm} Suppose that $t>r\geq2$ and $(r-1)^r\,\big|\,n-r$. Then
\begin{eqnarray*}
{\rm{spex}}_r^{\rm{lin}}(n,{\text{{\rm{Berge}}-}}K_{3,t})\geq 2^{-\frac{2}{r}}\Big(\sqrt{1+\frac{(n-r)2^{r+1}}{r-1}}+1\Big)^{\frac{2}{r}}.
\end{eqnarray*}
\end{thm}

Finally, we study the structure property of the spectral extremal connected Berge-$K_{3,t}$-free linear $r$-graphs and propose a conjecture about the structure of the spectral extremal connected Berge-$K_{3,t}$-free linear $r$-graphs.

\section{\normalsize A stability result for Berge-$K_{3,t}$ linear $r$-graphs}
\ \ \ \
Let $n\geq r\geq3$ and $H=(V(H),E(H))$ be a linear $r$-graph with $n$ vertices. 
For any two adjacent vertices $u,w\in V(H)$, the edge containing vertices $u$ and $w$ is unique by linearity of $H$, denoted by $l_{uw}$. Note that $l_{uw}=\emptyset$ when $u$ is not adjacent to $w$. For $u_1,\cdots,u_k\in V(H)$, let $N_{u_1\cdots u_k}=N_{u_1}(H)\cap\cdots\cap N_{u_k}(H)$.
In the following we present a stability result for Berge-$K_{3,t}$ linear $r$-graphs.

\begin{thm} Suppose that $n\geq r\geq3$ and $H$ be a linear $r$-graphs with $n$ vertices. Let $u,w\in V(H)$ be any two adjacent vertices and $v\in N_{u}(H)\backslash\{w\}$. Suppose $W\subseteq N_{vuw}\backslash (l_{vu}\cup l_{vw}\cup l_{uw})$ and any two vertices $u_1,u_2$ in $W$ satisfy $l_{vu_1}\neq l_{vu_2}$. If $|W|\geq (t-1)(r-1)+1$, then $H$ contains a Berge-$K_{3,t}$.
\end{thm}
\begin{proof}[{\bf{Proof}}]
Denote by $e_1,\cdots,e_s$ all edges in $E_{u}(H)$ intersecting with $W$. Since $W\subseteq N_{u}(H)$, $W=\bigcup_{i=1}^s(e_i\cap W)$. For any $e\in E_{u}(H)$, since $u\notin W$ and $u\in e$, we have $|e\cap W|\leq r-1$.
The assumption $|W|\geq (t-1)(r-1)+1$ implies that in $E_{u}(H)$ there exist at least $\lceil\frac{(t-1)(r-1)+1}{r-1}\rceil=t$ distinct edges intersecting with $W$. Thus $t\leq s\leq |W|$.

Note that $1\leq|e_i\cap W|\leq r-1$ for any $1\leq i\leq s$. For any $w'\in e_i\cap W$, we have $1\leq|l_{ww'}\cap W|\leq r-1$ for any $1\leq i\leq s$.
Let
$$E'_{u}(H)=\big\{e_i\,\big|\,\text{there\ exists\ some}\ y_i\in e_{i}\cap W\ \text{such\ that}\ l_{wy_i}\cap W=\{y_i\}\big\},$$
and $|E'_{u}(H)|=k$. Without loss of generality, we may set $E'_{u}(H)=\{e_1,\cdots,e_k\}$. Then for $e\in \{e_{k+1},\cdots,e_s\}$, we have $|l_{ww'}\cap W|\geq2$ for any $w'\in e\cap W$.

{\bf{Case 1.}} $k\geq t-1$.

Note that $y_i\in e_i$ for any $1\leq i\leq t-1$. Then $l_{uy_{1}},\cdots,l_{uy_{t-1}}$ are $t-1$ different edges. Since $l_{wy_i}\cap W=\{y_i\}$ for any $1\leq i\leq t-1$, $l_{wy_{1}},\cdots,l_{wy_{t-1}}$ are $t-1$ different edges.
Since $\big|W\backslash (e_1\cup\cdots\cup e_{t-1})\big|\geq|W|-(t-1)(r-1)\geq1$, we may select a vertex $y_{t}\in W\backslash (e_1\cup\cdots\cup e_{t-1})$. Then $l_{uy_{1}},\cdots,l_{uy_{t-1}},l_{uy_{t}}$ and $l_{wy_{1}},\cdots,l_{wy_{t-1}},l_{wy_{t}}$ are $t$ different edges, respectively.

The assumption $l_{vu_1}\neq l_{vu_2}$ for any $u_1,u_2\in W$ implies that $l_{vy_{1}},\cdots,l_{vy_{t-1}},l_{vy_{t}}$ are $t$ different edges. Note that $W\cap (l_{vu}\cup l_{vw}\cup l_{uw})=\emptyset$ by assumption. Then $l_{vy_{1}},\cdots,l_{vy_{t}}$, $l_{uy_{1}},\cdots,l_{uy_{t}}$, $l_{wy_{1}},\cdots,l_{wy_{t}}$ are $3t$ different edges and form a Berge-$K_{3,t}$ in $H$.

{\bf{Case 2.}} $k\leq t-2$.

Let $W'=(e_{k+1}\cup\cdots\cup e_{s})\cap W$. Then $|W'|\geq (t-k-1)(r-1)+1$. Without loss of generality, we may assume that $|e_{k+1}\cap W|\leq\cdots\leq|e_s\cap W|$. In the following, we give a rule for selecting vertices from $e_{k+1},\cdots,e_{s}$ in turn to form a skeleton of Berge-$K_{3,t}$. Since $|e_{k+1}\cap W'|\geq1$, we take one vertex $y_{k+1}\in e_{k+1}\cap W'$ that satisfies
\begin{eqnarray*}
|l_{wy_{k+1}}\cap W'|=\min_{u'\in e_{k+1}\cap W}|l_{wu'}\cap W'|,
\end{eqnarray*}
and for some $k+2\leq q\leq s$,
\begin{eqnarray*}
|l_{wy_{k+1}}\cap [(e^*_{{k+2}}\cup\cdots\cup e^*_{{p}})\cap W]|=\max_{u'\in e_{k+1}\cap W}|l_{wu'}\cap [(e^*_{{k+2}}\cup\cdots\cup e^*_{{p}})\cap W]|
\end{eqnarray*}
for any $k+2\leq p\leq q$, where $\{e^*_{k+2},\cdots,e^*_{s}\}=\{e_{k+2},\cdots,e_{s}\}$ satisfies $|e^*_{k+2}\cap W'|\geq\cdots\geq|e^*_{s}\cap W'|$.

For $k+2\leq i\leq s$, suppose we may take the vertices $y_{k+1},\cdots,y_{k+a}$ from $e_{k+1},\cdots,e_{i-1}$ in turn such that
\begin{eqnarray}
y_{j}\in (l_{uy_j}\cap W')\big\backslash (l_{wy_{k+1}}\cup\cdots\cup l_{wy_{j-1}}),
\end{eqnarray}
\begin{eqnarray}
\big|l_{wy_{j}}\cap W'_{j-1}\big|=\min_{u'\in l_{uy_j}\cap W' \atop u'\notin l_{wy_{k+1}}\cup\cdots\cup l_{wy_{j-1}}}\big|l_{wu'}\cap W'_{j-1}\big|,
\end{eqnarray}
and for some $j+1\leq q\leq s$,
\begin{eqnarray}
|l_{wy_{j}}\cap [(e^*_{{j+1}}\cup\cdots\cup e^*_{{p}})\cap W]|=\max_{u'\in l_{uy_j}\cap W' \atop u'\notin l_{wy_{k+1}}\cup\cdots\cup l_{wy_{j-1}}}|l_{wu'}\cap [(e^*_{{j+1}}\cup\cdots\cup e^*_{{p}})\cap W]|
\end{eqnarray}
for any $j+1\leq p\leq q$ and $k+1\leq j\leq k+a$, where $1\leq a\leq i-k-1$, $W'_{j-1}=W'\big\backslash (l_{uy_{k+1}}\cup\cdots\cup l_{uy_{j-1}})$ for $k+1\leq j\leq k+a$, and $\{e^*_{j+1},\cdots,e^*_{s}\}=\{e_{j+1},\cdots,e_{s}\}$ satisfies $|(e^*_{j+1}\cap W')\backslash (l_{wy_{k+1}}\cup\cdots\cup l_{wy_{j-1}})|\geq\cdots\geq|(e^*_{s}\cap W')\backslash (l_{wy_{k+1}}\cup\cdots\cup l_{wy_{j-1}})|$.

Now, we consider $e_{i}$, where $k+2\leq i\leq s$. If $|e_i\cap W'|\geq a+1$, then we have $(e_i\cap W')\backslash (l_{wy_{k+1}}\cup\cdots\cup l_{wy_{k+a}})\neq\emptyset$ by linearity of $H$. We take one vertex $y_{k+a+1}$ such that
\begin{eqnarray*}
y_{k+a+1}\in (e_i\cap W')\big\backslash (l_{wy_{k+1}}\cup\cdots\cup l_{wy_{k+a}}),
\end{eqnarray*}
\begin{eqnarray}
\big|l_{wy_{k+a+1}}\cap W'_{k+a}\big|=\min_{u'\in e_i\cap W' \atop u'\notin l_{wy_{k+1}}\cup\cdots\cup l_{wy_{k+a}}}\big|l_{wu'}\cap W'_{k+a}\big|,
\end{eqnarray}
and for some $k+a+2\leq q\leq s$,
\begin{small}
\begin{eqnarray}
|l_{wy_{k+a+1}}\cap [(e^*_{{k+a+2}}\cup\cdots\cup e^*_{{p}})\cap W]|=\max_{u'\in e_i\cap W' \atop u'\notin l_{wy_{k+1}}\cup\cdots\cup l_{wy_{k+a}}}|l_{wu'}\cap [(e^*_{{k+a+2}}\cup\cdots\cup e^*_{{p}})\cap W]|
\end{eqnarray}
\end{small}
for any $k+a+2\leq p\leq q$, where $W'_{k+a}=W'\big\backslash (l_{uy_{k+1}}\cup\cdots\cup l_{uy_{k+a}})$ and $\{e^*_{k+a+2},\cdots,e^*_{s}\}=\{e_{k+a+2},\cdots,e_{s}\}$ satisfies $|(e^*_{k+a+2}\cap W')\backslash (l_{wy_{k+1}}\cup\cdots\cup l_{wy_{k+a}})|\geq\cdots\geq|(e^*_{s}\cap W')\backslash (l_{wy_{k+1}}\cup\cdots\cup l_{wy_{k+a}})|$.
Then $l_{uy_{k+1}},\cdots,l_{uy_{k+a+1}}$ and $l_{wy_{k+1}},\cdots,l_{wy_{k+a+1}}$ are $k+a+1$ distinct edges, respectively.

If $|e_i\cap W'|\leq a$ and $(e_i\cap W')\big\backslash (l_{wy_{k+1}}\cup\cdots\cup l_{wy_{k+a}})\neq\emptyset$, then we take one vertex $y_{k+a+1}$ in $(e_i\cap W')\big\backslash (l_{wy_{k+1}}\cup\cdots\cup l_{wy_{k+a}})$ which satisfies (2.4) and (2.5). Otherwise, we have $w'\in l_{wy_{k+1}}\cup\cdots\cup l_{wy_{k+a}}$ for any $w'\in e_i\cap W'$. Then we continue to consider $e_{i+1}$ when $i+1\leq s$.

By the above rule for selecting vertices, suppose we have chosen $b$ vertices $y_{k+1},\cdots,y_{k+b}$ in $e_{k+1},\cdots,e_{s}$, where $k+1\leq k+b\leq s$.

{\bf{Case 2.1.}} $k+b\geq t$.

According to (2.1), $l_{uy_{k+1}},\cdots,l_{uy_{t}}$ and $l_{wy_{k+1}},\cdots,l_{wy_{t}}$ are $t$ different edges, respectively.
Then $l_{uy_{1}},\cdots,l_{uy_{k}},l_{uy_{k+1}},\cdots,l_{uy_{t}}$ and $l_{wy_{1}},\cdots,l_{wy_{k}},l_{wy_{k+1}},\cdots,l_{wy_{t}}$ are $t$ different edges, respectively.  The assumption that $l_{vu_1}\neq l_{vu_2}$ for any $u_1,u_2\in W$, $l_{vy_{1}},\cdots,l_{vy_{t}}$ are $t$ different edges. Thus, $l_{vy_{1}},\cdots,l_{vy_{t}}$, $l_{uy_{1}},\cdots,l_{uy_{t}}$, $l_{wy_{1}},\cdots,l_{wy_{t}}$ are $3t$ different edges and form a Berge-$K_{3,t}$ in $H$ by $W\cap (l_{vu}\cup l_{vw}\cup l_{uw})=\emptyset$.

{\bf{Case 2.2.}} $k+b\leq t-1$.

If we cannot continue to select a vertex $y_{k+b+1}$ satisfying (2.1)-(2.3), then $\sum_{j=k+1}^s|e_{j}\cap W|\leq b(r-1)$. Suppose to the contrary that $\sum_{j=k+1}^s|e_{j}\cap W|>b(r-1)$. Without loss of generality, we may assume that $x_{k+1},\cdots,x_{k+b}\in W'$ such that $l_{vy_{1}},\cdots,l_{vy_{k}}$, $l_{vx_{k+1}},\cdots,l_{vx_{k+b}}$, $l_{uy_{1}},\cdots,l_{uy_{k}}$, $l_{ux_{k+1}},\cdots,l_{ux_{k+b}}$, $l_{wy_{1}},\cdots,l_{wy_{k}}$, $l_{wx_{k+1}},\cdots,l_{wx_{k+b}}$ form a Berge-$K_{3,k+b}$, and there is no vertex $x_{k+b+1}\in W'$ such that $l_{vy_{1}},\cdots,l_{vy_{k}}$, $l_{vx_{k+1}},\cdots,l_{vx_{k+b+1}}$, $l_{uy_{1}},\cdots,l_{uy_{k}}$, $l_{ux_{k+1}},\cdots,l_{ux_{k+b+1}}$, $l_{wy_{1}},\cdots,l_{wy_{k}}$, $l_{wx_{k+1}},\cdots,l_{wx_{k+b+1}}$ form a Berge-$K_{3,k+b+1}$.
Since $\sum_{j=k+1}^s|e_{j}\cap W|>b(r-1)$, we have
\begin{eqnarray*}
\big|\bigcup_{j=k+1}^s(e_{j}\cap W)\big\backslash (l_{wx_{k+1}}\cup\cdots\cup l_{wx_{k+b}})\big|\geq \sum_{j=k+1}^s|e_{j}\cap W|-b(r-1)>0.
\end{eqnarray*}
Note that $k+b\leq t-1<s$. We claim that $\widetilde{w}\notin e\in \{e_{k},\cdots,e_s\}\backslash\{l_{ux_{k+1}},\cdots,l_{ux_{k+b}}\}$ for any $\widetilde{w}\in \bigcup_{j=k+1}^s(e_{j}\cap W)\big\backslash (l_{wx_{k+1}}\cup\cdots\cup l_{wx_{k+b}})$. Otherwise, we may select the vertex $x_{k+b+1}:=\widetilde{w}$ such that $l_{vy_{1}},\cdots,l_{vy_{k}}$, $l_{vx_{k+1}},\cdots,l_{vx_{k+b+1}}$, $l_{uy_{1}},\cdots,l_{uy_{k}}$, $l_{ux_{k+1}},\cdots,l_{ux_{k+b+1}}$, $l_{wy_{1}},\cdots,l_{wy_{k}}$, $l_{wx_{k+1}},\cdots,l_{wx_{k+b+1}}$ form a Berge-$K_{3,k+b+1}$, which is a contradiction.
Thus, $\widetilde{w}\in l_{ux_{i}}$ for any $\widetilde{w}\in \bigcup_{j=k+1}^s(e_{j}\cap W)\big\backslash (l_{wx_{k+1}}\cup\cdots\cup l_{wx_{k+b}})$, where $k+1\leq i\leq k+b$.
By the linearity of $H$, since $1\leq|e_{k+1}\cap W|\leq\cdots\leq|e_s\cap W|\leq r-1$, $\sum_{j=k+1}^s|e_{j}\cap W|>b(r-1)$ and the assumption that there is no vertex $x_{k+b+1}\in W'$ such that a Berge-$K_{3,k+b+1}$ is formed, we may find at least one vertex $\widetilde{w}$ in $\bigcup_{j=k+1}^s(e_{j}\cap W)\big\backslash (l_{wx_{k+1}}\cup\cdots\cup l_{wx_{k+b}})$ satisfying one of the following two cases:
({\rm{a}}) for some $k+1\leq j\leq k+b$, one has $\widetilde{w}\in l_{ux_{j}}$ and $\big|l_{w\widetilde{w}}\cap W'_{j-1}\big|<\big|l_{wx_{j}}\cap W'_{j-1}\big|$;
({\rm{b}}) for some $k+1\leq j\leq k+b$, one has $\widetilde{w}\in l_{ux_{j}}$, $\big|l_{w\widetilde{w}}\cap W'_{j-1}\big|=\big|l_{wx_{j}}\cap W'_{j-1}\big|$ and $|l_{w\widetilde{w}}\cap [(e^*_{{j+1}}\cup\cdots\cup e^*_{{p}})\cap W]|>|l_{wx_{j}}\cap [(e^*_{{j+1}}\cup\cdots\cup e^*_{{p}})\cap W]|$ for some $j+1\leq p\leq q$, where $j+1\leq q\leq s$ and $\{e^*_{j+1},\cdots,e^*_{s}\}=\{e_{j+1},\cdots,e_{s}\}$ satisfies $|(e^*_{j+1}\cap W')\backslash (l_{wx_{k+1}}\cup\cdots\cup l_{wx_{j-1}})|\geq\cdots\geq|(e^*_{s}\cap W')\backslash (l_{wx_{k+1}}\cup\cdots\cup l_{wx_{j-1}})|$. This implies that $x_j$ does not satisfy condition (2.2) or (2.3) for some $k+1\leq j\leq k+b$ (two examples are shown in Remarks 2.2 and 2.3).

Similar to Case 2.1, $l_{vy_{1}},\cdots,l_{vy_{k+b}}$, $l_{uy_{1}},\cdots,l_{uy_{k+b}}$, $l_{wy_{1}},\cdots,l_{wy_{k+b}}$ are $3(k+b)$ different edges and form a Berge-$K_{3,k+b}$ in $H$. The fact that we cannot continue to select a vertex $y_{k+b+1}$ satisfying (2.1)-(2.3) is equivalent to the fact that there is no vertex $y_{k+b+1}\in W'$ such that $l_{vy_{1}},\cdots,l_{vy_{k+b+1}}$, $l_{uy_{1}},\cdots,l_{uy_{k+b+1}}$, $l_{wy_{1}},\cdots,l_{wy_{k+b+1}}$ form a Berge-$K_{3,k+b+1}$.
Therefore, by the above analysis of $x_{k+1},\cdots,x_{k+b}$, we may obtain that there exists some $k+1\leq j\leq k+b$ such that $y_j$ does not satisfy condition (2.2) or (2.3), which is a contradiction.

So we have $\sum_{j=k+1}^s|e_{j}\cap W|\leq b(r-1)$. While
\begin{eqnarray*}
b(r-1)<(t-k-1)(r-1)+1\leq |W'|=\sum_{j=k+1}^s|e_{j}\cap W|,
\end{eqnarray*}
which is a contradiction.

Thus, we may continue to select a vertex $y_{k+b+1}$ that satisfy (2.1)-(2.3) when $k+b\leq t-1$. That is, we may take at least $t-k$ vertices $y_{k+1},\cdots,y_{t}$ such that $l_{uy_{k+1}},\cdots,l_{uy_{t}}$ and $l_{wy_{k+1}},\cdots,l_{wy_{t}}$ are $t$ different edges, respectively. Then $l_{uy_{1}},\cdots,l_{uy_{t}}$ and $l_{wy_{1}},\cdots,l_{wy_{t}}$ are $t$ different edges, respectively. Similarly, $l_{vy_{1}},\cdots,l_{vy_{t}}$, $l_{uy_{1}},\cdots,l_{uy_{t}}$, $l_{wy_{1}},\cdots,l_{wy_{t}}$ are $3t$ different edges and form a Berge-$K_{3,t}$ in $H$.
\end{proof}

\begin{figure}[h]
  \centering
  \includegraphics[scale=0.6]{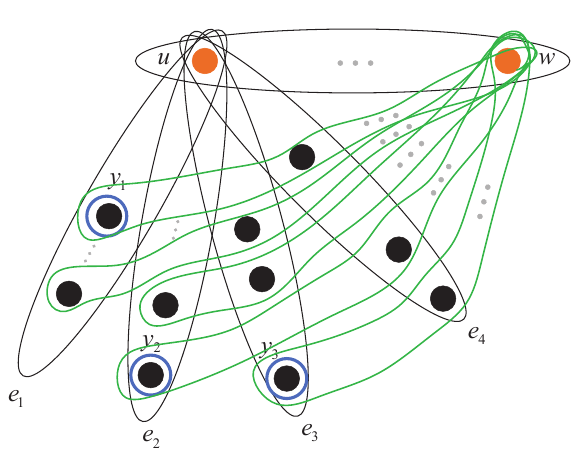}\\
  \caption{\small{Example of Remark 2.2.}}\label{fig1}
\end{figure}

\begin{rem}
Note that the restriction of (2.2) on $y_j$ is necessary. Otherwise, we may cannot select $y_{k+b+1}$ when $k+b\leq t-1$. For an example: when $k=0$, $r=t=4$ and $|W|=(t-1)(r-1)+1=10$, if we ignore the restriction of (2.2) on $y_3$, then we cannot select $y_4$ in this $4$-graph (see Figure 1), where the vertices in $W$ are represented by black dots.
\end{rem}

\begin{figure}[h]
  \centering
  \includegraphics[scale=0.6]{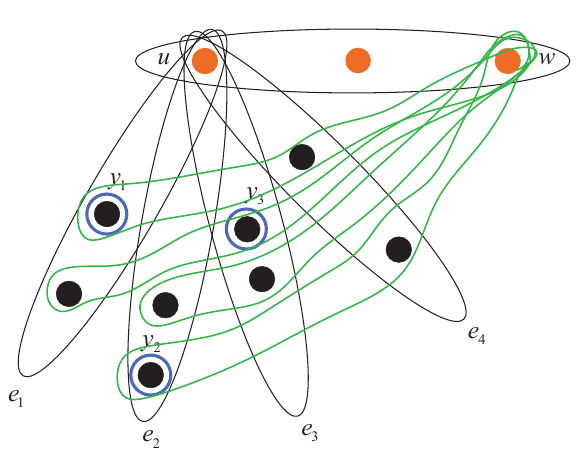}\\
  \caption{\small{Example of Remark 2.3.}}\label{fig2}
\end{figure}

\begin{rem}
Note that the restriction of (2.3) on $y_j$ is necessary. Otherwise, we may cannot select $y_{k+b+1}$ when $k+b\leq t-1$. For an example: when $k=0$, $b=r=3$, $s=t=4$ and $|W|=8>7=(t-1)(r-1)+1$, if we ignore the restriction of (2.3) on $y_2$, then we cannot select $y_4$ in this $3$-graph (see Figure 2), where the vertices in $W$ are represented by black dots.
\end{rem}

\section{\normalsize Linear Tur\'{a}n number of Berge-$K_{3,t}$}
\ \ \ \ \
In this section, we prove an upper bound for the linear Tur\'{a}n number of Berge-$K_{3,t}$ by using the stability result (see Theorem 2.1) of Section 2.
Let $\mathcal{H}_{3,t}^{n,r}$ be the set of connected Berge-$K_{3,t}$-free linear $r$-graphs wtih $n$ vertices, where $t,r\geq3$.

For any two adjacent vertices $u,w\in V(H)$, set
\begin{eqnarray*}
N_1(u,w)&=&\big\{v\,\big|\,v\in N_{uw}\ \text{and}\ v\notin l_{uw}\big\}, \\
N_2(u,w)&=&\big\{v\,\big|\,v\in N_u(H)\backslash(N_w(H)\cup \{w\})\big\}.
\end{eqnarray*}
Then $N_1(u,w)$, $N_2(u,w)$ and $l_{uw}\backslash\{u\}$ are pairwisely disjoint and $$N_{u}(H)=N_1(u,w){\cup} N_2(u,w){\cup} l_{uw}\backslash\{u\}.$$
For a set of vertices $U\subseteq V(H)$, let $E^k(U)=\{e\,\big|\,|e\cap U|=k, e\in E(H)\}$ and $E_v^k(U)=\{e\,\big|\,|e\cap U|=k, v\in e\in E(H)\}$. The cardinalities of $E^k(U)$ and $E_v^k(U)$ are denoted by $e^k(U)$ and $e_v^k(U)$, respectively.

\begin{lem} Suppose that $n\geq r\geq3$ and $H$ is a linear $r$-graph wtih $n$ vertices. Let $u,w\in V(H)$ be any two adjacent vertices. Then
\begin{kst}
\item[(\romannumeral1)] $e^0_v(N_{uw})=0$ for any $v\in N_1(u,w)$ or $v\in l_{uw}\backslash\{u,w\}$, and $e^r_v(N_{uw})=0$ for any $v\in N_2(u,w)$;
\item[(\romannumeral2)] $$\small\sum_{v\in N_{u}(H)}d_v=\sum_{v\in N_1(u,w)}\sum_{i=1}^{r}e^i_v(N_{uw})+\sum_{v\in N_2(u,w)}\sum_{i=0}^{r-1}e^i_v(N_{uw})+\sum_{v\in l_{uw}\backslash\{u,w\}}\sum_{i=1}^{r}e^i_v(N_{uw})+d_{w}.$$
\end{kst}
\end{lem}
\begin{proof}[{\bf{Proof}}]
$(\romannumeral1)$. For any $v\in N_1(u,w)$ or $v\in l_{uw}\backslash\{u,w\}$, we have $v\in N_{uw}$. Then for $e\in E_v(H)$, we have $|e\cap N_{uw}|\geq1$. Thus $e^0_v(N_{uw})=0$. 
For any $v\in N_2(u,w)$, we have $v\notin N_{uw}$. Then for $e\in E_v(H)$, we have $|e\cap N_{uw}|\leq r-1$. Thus $e^r_v(N_{uw})=0$. 

$(\romannumeral2)$. Note that $d_v=\sum_{i=0}^{r}e^i_v(N_{uw})$ for any $v\in V(H)$. Combining the results of $(\romannumeral1)$, we have
\begin{small}
\begin{eqnarray*}
\sum_{v\in N_{u}(H)}d_v&=&\sum_{v\in N_{u}(H)}\sum_{i=0}^{r}e^i_v(N_{uw}) \\
&=&\sum_{v\in N_1(u,w)}\sum_{i=0}^{r}e^i_v(N_{uw})+\sum_{v\in N_2(u,w)}\sum_{i=0}^{r}e^i_v(N_{uw})+\sum_{v\in l_{uw}\backslash\{u,w\}}\sum_{i=0}^{r}e^i_v(N_{uw})+d_w \\
&=&\sum_{v\in N_1(u,w)}\sum_{i=1}^{r}e^i_v(N_{uw})+\sum_{v\in N_2(u,w)}\sum_{i=0}^{r-1}e^i_v(N_{uw})+\sum_{v\in l_{uw}\backslash\{u,w\}}\sum_{i=1}^{r}e^i_v(N_{uw})+d_{w}.
\end{eqnarray*}
\end{small}
\end{proof}

\begin{lem} Let $n\geq r\geq3$ and $u,w\in V(H)$ be any two adjacent vertices, where $H\in \mathcal{H}_{3,t}^{n,r}$.
\begin{kst}
\item[(\romannumeral1)] For any $v\in N_1(u,w)$, we have $\sum_{i=2}^{r}e^i_v(N_{uw})\leq t(r-1)+1$;
\item[(\romannumeral2)] For any $v\in N_2(u,w)$, we have $\sum_{i=1}^{r-1}e^i_v(N_{uw})\leq t(r-1)$;
\item[(\romannumeral3)] For any $v\in l_{uw}\backslash\{u,w\}$, we have $\sum_{i=2}^{r}e^i_v(N_{uw})\leq (t-1)(r-1)+1$.
\end{kst}
\end{lem}
\begin{proof}[{\bf{Proof}}]
$(\romannumeral1)$. Suppose to the contrary  that there is a vertex $v\in N_1(u,w)$ such that $\sum_{i=2}^{r}e^i_v(N_{uw})\geq t(r-1)+2$. Then we have
$$\big|\bigcup_{i=2}^{r}E^i_v(N_{uw})\backslash \{l_{vu},l_{vw}\}\big|\geq \sum_{i=2}^{r}e^i_v(N_{uw})-2\geq t(r-1).$$

Since $H$ is linear, there are at most $r-2$ edges in $\bigcup_{i=2}^{r}E^i_v(N_{uw})\backslash \{l_{vu},l_{vw}\}$ that intersect with $l_{uw}$.
Then from $\bigcup_{i=2}^{r}E^i_v(N_{uw})\backslash \{l_{vu},l_{vw}\}$, we may take $t(r-1)-(r-2)=(t-1)(r-1)+1$ edges disjoint from $l_{uw}$.
Noting that for any $e\in E^i_v(N_{uw})\backslash \{l_{vu},l_{vw}\}$, we have $e\cap l_{vu}=\{v\}$ and $e\cap l_{vw}=\{v\}$. Then the edge $e$ contains $i-1\geq1$ vertices in $N_{vuw}\backslash (l_{vu}\cup l_{vw})$.
We may take a vertex in $N_{vuw}\backslash (l_{vu}\cup l_{vw})$ from each of these $(t-1)(r-1)+1$ edges, and denote the set of these vertices by $W$. Then $|W|=(t-1)(r-1)+1$. Since $v\notin l_{uw}$, $l_{uw}\notin \bigcup_{i=2}^{r}E^i_v(N_{uw})$ and $W\cap l_{uw}=\emptyset$. Then $W\subseteq N_{vuw}\backslash (l_{vu}\cup l_{vw}\cup l_{uw})$. Clearly, for any $u_1,u_2\in W$, we have $l_{vu_1}\neq l_{vu_2}$. By Theorem 2.1, $H$ contains a Berge-$K_{3,t}$, which is a contradiction.

$(\romannumeral2)$. Suppose to the contrary  that there is a vertex $v\in N_2(u,w)$ such that $\sum_{i=1}^{r-1}e^i_v(N_{uw})\geq t(r-1)+1$. Then we have
$$\big|\bigcup_{i=1}^{r-1}E^i_v(N_{uw})\backslash \{l_{vu}\}\big|\geq \sum_{i=1}^{r-1}e^i_v(N_{uw})-1\geq t(r-1).$$

Since $H$ is linear, there are at most $r-2$ edges in $\bigcup_{i=1}^{r-1}E^i_v(N_{uw})\backslash \{l_{vu}\}$ that intersect with $l_{uw}$.
Then from $\bigcup_{i=1}^{r-1}E^i_v(N_{uw})\backslash \{l_{vu}\}$, we may take $t(r-1)-(r-2)=(t-1)(r-1)+1$ edges disjoint from $l_{uw}$.
Noting that for any $e\in E^i_v(N_{uw})\backslash \{l_{vu}\}$, we have $e\cap l_{vu}=\{v\}$. Then the edge $e$ contains $i\geq1$ vertices in $N_{vuw}\backslash l_{vu}$.
We may take a vertex in $N_{vuw}\backslash l_{vu}$ from each of these $(t-1)(r-1)+1$ edges, and denote the set of these vertices by $W$. Then $|W|=(t-1)(r-1)+1$. Since $v\notin l_{uw}$, $l_{uw}\notin \bigcup_{i=1}^{r-1}E^i_v(N_{uw})$ and $W\cap l_{uw}=\emptyset$. Then $W\subseteq N_{vuw}\backslash (l_{vu}\cup l_{uw})$. Noting that for $v\in N_2(u,w)$, we have $l_{vw}=\emptyset$ and $N_{vuw}\backslash (l_{vu}\cup l_{vw}\cup l_{uw})=N_{vuw}\backslash (l_{vu}\cup l_{uw})$. Clearly, for any $u_1,u_2\in W$, we have $l_{vu_1}\neq l_{vu_2}$. By Theorem 2.1, $H$ contains a Berge-$K_{3,t}$, which is a contradiction.

$(\romannumeral3)$. Suppose to the contrary  that there is a vertex $v\in l_{uw}\backslash\{u,w\}$ such that $\sum_{i=2}^{r}e^i_v(N_{uw})\geq (t-1)(r-1)+2$. Then we have
$$\big|\bigcup_{i=2}^{r}E^i_v(N_{uw})\backslash \{l_{uw}\}\big|\geq \sum_{i=2}^{r}e^i_v(N_{uw})-1\geq (t-1)(r-1)+1.$$

Since $H$ is linear, for any $e\in\bigcup_{i=2}^{r}E^i_v(N_{uw})\backslash \{l_{uw}\}$, we have $e\cap l_{uw}=\{v\}$. Then the edge $e$ contains $i-1\geq1$ vertices in $N_{vuw}\backslash l_{uw}$. From $\bigcup_{i=2}^{r}E^i_v(N_{uw})\backslash \{l_{uw}\}$, we may take $(t-1)(r-1)+1$ edges.
We may take a vertex in $N_{vuw}\backslash l_{uw}$ from each of these $(t-1)(r-1)+1$ edges, and denote the set of these vertices by $W$. Then $W\subseteq N_{vuw}\backslash l_{uw}$ and $|W|=(t-1)(r-1)+1$.
Noting that for $v\in l_{uw}\backslash\{u,w\}$, we have $l_{vu}=l_{vw}=l_{uw}$ and $N_{vuw}\backslash (l_{vu}\cup l_{vw}\cup l_{uw})=N_{vuw}\backslash l_{uw}$. Clearly, for any $u_1,u_2\in W$, we have $l_{vu_1}\neq l_{vu_2}$. By Theorem 2.1, $H$ contains a Berge-$K_{3,t}$, which is a contradiction.
\end{proof}

Recall that for $n\geq r\geq2$, $t\geq3$ and $x,y\in \mathbb{R}$, we have
\begin{eqnarray*}
f(n,r,t,x,y)&:=& \bigg(t(r-1)+\frac{(r-2)(n-r)}{2(r-1)}+\frac{n-r}{(r-1)^2}\bigg)x+\frac{1}{r-1}y \notag\\
&{}&-\bigg(t+r-2+\frac{(r-2)(n-r-2)}{2(r-1)}-\frac{(r-2)(r-3)(n-r)}{2(r-1)^2}\bigg).
\end{eqnarray*}

\begin{lem} Let $n\geq r\geq3$ and $u,w\in V(H)$ be any two adjacent vertices, where $H\in \mathcal{H}_{3,t}^{n,r}$. Then
\begin{eqnarray*}
\sum_{v\in N_{u}(H)\backslash \{w\}}d_v\leq (r-1)f(n,r,t,d_u,d_w)-d_w.
\end{eqnarray*}
\end{lem}
\begin{proof}[{\bf{Proof}}]
Let $e_i$ be all edges containing $u$, where $0\leq i\leq d_{u}-1$ and $e_0=l_{uw}$. Set
\begin{eqnarray*}
L_{v}&=&
\begin{cases}
\big\{w'\, \big|\, w'\in e\in E^1_v(N_{uw})\backslash\{l_{uv}\}, w'\neq v\big\},\ \ \text{if}\ v\in N_1(u,w), \\
\big\{w'\, \big|\, w'\in e\in E^0_v(N_{uw})\backslash\{l_{uv}\}, w'\neq v\big\},\ \ \text{if}\ v\in N_2(u,w), \\
\big\{w'\, \big|\, w'\in e\in E^1_v(N_{uw}), w'\neq v\big\},\ \ \text{if}\ v\in e_0\backslash\{u,w\},
\end{cases} \\
L_0&=&\bigcup_{v\in e_0\backslash \{u,w\}}L_{v}, \\
L_i&=&\bigcup_{v\in e_i\backslash \{u\}}L_{v},\ \ 1\leq i\leq d_{u}-1.
\end{eqnarray*}

Since $H$ is linear, by Lemmas 3.1$(\romannumeral1)$ and 3.2, we have
\begin{eqnarray*}
e^1_v(N_{uw})=d_v-e^0_v(N_{uw})-\sum_{i=2}^re^i_v(N_{uw})=d_v-\sum_{i=2}^re^i_v(N_{uw})\geq d_v-\big(t(r-1)+1\big)
\end{eqnarray*}
for any $v\in N_1(u,w)$. Then
\begin{eqnarray*}
|L_{v}|=(r-1)e^1_v(N_{uw})\geq(r-1)\big(d_v-t(r-1)-1\big)
\end{eqnarray*}
for any $v\in N_1(u,w)$. Similarly, we have
\begin{gather*}
|L_{v}|=(r-1)e^0_v(N_{uw})\geq (r-1)\big(d_v-t(r-1)\big),\ v\in N_2(u,w), \\
|L_{v}|=(r-1)e^1_v(N_{uw})\geq (r-1)\big(d_v-(t-1)(r-1)-1\big),\ v\in e_0\backslash\{u,w\}.
\end{gather*}
Thus
\begin{eqnarray}
\sum_{v\in N_{u}(H)\backslash \{w\}}|L_v|&\geq& \sum_{v\in N_1(u,w)}(r-1)\big(d_v-t(r-1)-1\big)+\sum_{v\in N_2(u,w)}(r-1)\big(d_v-t(r-1)\big) \notag\\
&{}& +\sum_{v\in e_0\backslash\{u,w\}}(r-1)\big(d_v-(t-1)(r-1)-1\big) \notag\\
&\geq& (r-1)\Big(\sum_{v\in N_{u}(H)\backslash \{w\}}d_v\Big)-(r-1)(r-2)\big((t-1)(r-1)+1\big) \notag\\
&{}& -(r-1)\big((r-1)d_{u}-r+1\big)\big(t(r-1)+1\big).
\end{eqnarray}

Let $e_i=\{u,v_{2,{i}},\cdots,v_{r,i}\}$, where $1\leq i\leq d_{u}-1$. For $2\leq p<q\leq r$, if $v_{p,i},v_{q,i}\notin N_{uw}$, then $(N_{uw}\cup\{u,w\})\cap (L_{v_{p,i}}\cap L_{v_{q,i}})=\emptyset$, $(e_i\backslash\{u\})\cap (L_{v_{p,i}}\cap L_{v_{q,i}})=\emptyset$ and $|N_{uw}|\geq r-2$. Thus
\begin{eqnarray*}
\big|L_{v_{p,i}}\cap L_{v_{q,i}}\big|\leq n-\big|N_{uw}\big|-2-(r-1)+\big|N_{uw}\cap \big(e_i\backslash\{u,v_{p,i},v_{q,i}\}\big)\big|\leq n-r-2.
\end{eqnarray*}
If $v_{p,i}\in N_{uw}, v_{q,i}\notin N_{uw}$, then $|N_{uw}|\geq r-1$. Therefore
\begin{eqnarray*}
\big|L_{v_{p,i}}\cap L_{v_{q,i}}\big|\leq n-\big|N_{uw}\big|-2-(r-2)+\big|N_{uw}\cap \big(e_i\backslash\{u,v_{p,i},v_{q,i}\}\big)\big|\leq n-r-2.
\end{eqnarray*}
If $v_{p,i},v_{q,i}\in N_{uw}$, then $|N_{uw}|\geq r$. Then we get
\begin{eqnarray*}
\big|L_{v_{p,i}}\cap L_{v_{q,i}}\big|\leq n-\big|N_{uw}\big|-2-(r-3)+\big|N_{uw}\cap \big(e_i\backslash\{u,v_{p,i},v_{q,i}\}\big)\big|\leq n-r-2.
\end{eqnarray*}
Then for $1\leq i\leq d_{{u}-1}$, we have
\begin{eqnarray*}
\big|L_i\big|=\big|\bigcup_{v\in e_i\backslash \{u\}}L_{v}\big|=\big|\bigcup_{j=2}^{r}L_{v_{j,i}}\big|&\geq& \sum_{j=2}^{r}\big|L_{v_{j,i}}\big|-\sum_{2\leq p<q\leq r}\big|L_{v_{p,i}}\cap L_{v_{q,i}}\big| \notag\\
&\geq& \sum_{j=2}^{r}\big|L_{v_{j,i}}\big|-\binom{r-1}{2}(n-r-2).
\end{eqnarray*}
Thus
\begin{eqnarray}
\sum_{i=1}^{d_u-1}\sum_{j=2}^{r}\big|L_{v_{j,i}}\big|\leq \sum_{i=1}^{d_u-1}\big|L_i\big|+(d_u-1)\binom{r-1}{2}(n-r-2).
\end{eqnarray}

Let $e_0=\{u,w,v_{3,0},\cdots,v_{r,0}\}$. For $3\leq p<q\leq r$, we have $|L_{v_{p,0}}\cap L_{v_{q,0}}|\leq n-|N_{uw}|-2\leq n-r$ and
\begin{eqnarray*}
\big|L_0\big|=\big|\bigcup_{v\in e_0\backslash \{u,w\}}L_{v}\big|=\big|\bigcup_{j=3}^{r}L_{v_{j,0}}\big|&\geq& \sum_{j=3}^{r}\big|L_{v_{j,0}}\big|-\sum_{3\leq p<q\leq r}\big|L_{v_{p,0}}\cap L_{v_{q,0}}\big| \notag\\
&\geq& \sum_{j=3}^{r}\big|L_{v_{j,0}}\big|-\binom{r-2}{2}(n-r),
\end{eqnarray*}
that is
\begin{eqnarray}
\sum_{j=3}^{r}\big|L_{v_{j,0}}\big|\leq \big|L_0\big|+\binom{r-2}{2}(n-r).
\end{eqnarray}

Note that $N_{u}(H)\backslash \{w\}=\{v_{3,0},\cdots,v_{r,0}\}\cup\{v_{2,i},\cdots,v_{r,i}\, |\, i=1,\cdots,d_u-1\}$. By (3.2) and (3.3), we obtain
\begin{eqnarray}
\sum_{v\in N_{u}(H)\backslash \{w\}}\big|L_v\big|\leq \sum_{i=0}^{d_{u}-1}\big|L_i\big|+(d_{u}-1)\binom{r-1}{2}(n-r-2)+\binom{r-2}{2}(n-r).
\end{eqnarray}

By the definition of $L_v$ and $L_i$, we have $(N_{uw}\cup\{u,w\})\cap L_i=\emptyset$ and $e_i\cap L_i=\emptyset$ for $0\leq i\leq d_{u}-1$. Then
\begin{gather*}
\big|L_i\big|\leq n-\big|N_{uw}\big|-2-(r-1)+\big|N_{uw}\cap e_i\big|\leq n-2r+1,\ 1\leq i\leq d_{u}-1, \\
\big|L_0\big|\leq n-\big|N_{uw}\big|-2-(r-2)+\big|N_{uw}\cap e_0\big|=n-\big|N_{uw}\big|-2\leq n-r.
\end{gather*}
Hence, we get
\begin{eqnarray}
\sum_{i=0}^{d_{u}-1}\big|L_i\big|\leq \sum_{i=1}^{d_{u}-1}(n-2r+1)+(n-r)\leq (d_{u}-1)(n-2r+1)+n-r.
\end{eqnarray}

By (3.4) and (3.5), we obtain
\begin{small}
\begin{eqnarray}
\sum_{v\in N_{u}(H)\backslash \{w\}}\big|L_v\big|&\leq& (d_{u}-1)(n-2r+1)+n-r+(d_{u}-1)(n-r-2)\binom{r-1}{2}+(n-r)\binom{r-2}{2} \notag\\
&=& (d_{u}-1)\bigg(n-2r+1+\binom{r-1}{2}(n-r-2)\bigg)+(n-r)\Bigg(\binom{r-2}{2}+1\Bigg).
\end{eqnarray}
\end{small}

By the lower bound in (3.1) and the upper bound in (3.6) for $\sum_{v\in N_{u}(H)\backslash \{w\}}\big|L_v\big|$, we get
\begin{small}
\begin{eqnarray*}
(r-1)\Big(\sum_{v\in N_{u}(H)\backslash \{w\}}d_v\Big)&\leq& 2\binom{r-1}{2}\Big((t-1)(r-1)+1\Big)+(n-r)\Bigg(\binom{r-2}{2}+1\Bigg) \\
&{}& +(d_{u}-1)\Bigg((r-1)^2\big(t(r-1)+1\big)+\binom{r-1}{2}(n-r-2)+n-2r+1\Bigg),
\end{eqnarray*}
\end{small}
that is
\begin{eqnarray*}
\sum_{v\in N_{u}(H)\backslash \{w\}}d_v&\leq&(r-2)\big((t-1)(r-1)+1\big)+(d_{u}-1)(r-1)\big(t(r-1)+1\big)+\frac{n-r}{r-1} \\
&{}& +(d_{u}-1)\bigg(\frac{n-r}{r-1}-1+\frac{(r-2)(n-r-2)}{2}\bigg)+\binom{r-2}{2}\frac{n-r}{r-1} \\
&=& \bigg(t(r-1)^2+\frac{(r-2)(n-r)}{2}+\frac{n-r}{r-1}\bigg)d_{u}-(t+r-2)(r-1) \\
&{}& -\frac{(r-2)(n-r-2)}{2}+\frac{(r-2)(r-3)(n-r)}{2(r-1)} \\
&=& (r-1)f(n,r,t,d_u,d_w)-d_w.
\end{eqnarray*}
\end{proof}

\begin{proof}[{\bf{Proof of Theorem 1.1.}}]
Let $H\in \mathcal{H}_{3,t}^{n,r}$ and $u,w\in V(H)$ be any two adjacent vertices. Since $H$ is a connected linear $r$-graphs with $n$ vertices, we have the largest degree $\Delta(H)\leq \frac{n-1}{r-1}$ for any $u\in V(H)$. By Lemma 3.3, we get
\begin{eqnarray*}
\sum_{v\in N_{u}(H)}d_v\leq (r-1)f(n,r,t,d_u,d_w)\leq (r-1)f(n,r,t,d_u,\frac{n-1}{r-1}),
\end{eqnarray*}
where the last inequality holds because when $n,r,t$ and $d_u$ are fixed, the larger $d_{w}$ is, the larger $f(n,r,t,d_u,d_w)$ is.

Denote by $d(H)$ the average degree of $H$. Then
\begin{small}
\begin{eqnarray}
\sum_{u\in V(H)}\sum_{v\in N_{u}(H)}d_v&\leq& \sum_{u\in V(H)}(r-1)f(n,r,t,d_u,\frac{n-1}{r-1}) \notag\\
&=& n\bigg(\frac{n-1}{r-1}-(t+r-2)(r-1)-\frac{(r-2)(n-r-2)}{2}+\frac{(r-2)(r-3)(n-r)}{2(r-1)}\bigg) \notag\\
&{}& +\bigg(t(r-1)^2+\frac{(r-2)(n-r)}{2}+\frac{n-r}{r-1}\bigg)nd(H).
\end{eqnarray}
\end{small}

On the other hand, since $v\in N_{u}(H)$ if and only if $u\in N_{v}(H)$, we may reverse the sum to get
\begin{eqnarray}
\sum_{u\in V(H)}\sum_{v\in N_{u}(H)}d_v=\sum_{v\in V(H)}(r-1)d_v\cdot d_v\geq (r-1)nd(H)^2,
\end{eqnarray}
with the last inequality follows from the Cauchy-Schwarz inequality. By (3.7) and (3.8), we get
\begin{eqnarray*}
d(H)^2-g_1(n,r,t)d(H)-g_2(n,r,t)\leq0,
\end{eqnarray*}
where
\begin{gather*}
g_1(n,r,t):=t(r-1)+\frac{(r-2)(n-r)}{2(r-1)}+\frac{n-r}{(r-1)^2}, \\
g_2(n,r,t):=\frac{(r-2)(r-3)(n-r)+2(n-1)}{2(r-1)^2}-(t+r-2)-\frac{(r-2)(n-r-2)}{2(r-1)}.
\end{gather*}
This gives that
\begin{eqnarray*}
d(H)\leq \frac{(r(r-3)+4)n}{2(r-1)^2}+O(n^\frac{1}{2})+c(r,t),
\end{eqnarray*}
where $c(r,t)$ is some constant that depends only on $r$ and $t$.
Thus, we have
\begin{eqnarray*}
|E(H)|=\frac{nd(H)}{r}\leq \frac{(r(r-3)+4)n^2}{2r(r-1)^2}+O(n^\frac{3}{2}).
\end{eqnarray*}
\end{proof}

\section{\normalsize Spectral radius of Berge-$K_{3,t}$-free linear $r$-graphs}
\ \ \ \
In this section, we obtain an upper bound and a lower bound for the maximum spectral radius of connected Berge-$K_{3,t}$-free linear $r$-graphs. Meanwhile, we discuss the structure property of spectral extremal connected Berge-$K_{3,t}$-free linear $r$-graphs.

\begin{lem}[\cite{C4}] Let $H$ be a connected simple linear $r$-graph and $\rho(H)$ be the spectral radius of $H$. Let $u$ be the vertex with the maximum entry of the Perron vector of $H$. Then
\begin{eqnarray*}
\rho^2(H)\leq \frac{1}{r-1}\sum_{v\in N_u(H)}d_v.
\end{eqnarray*}
\end{lem}

\begin{proof}[{\bf{Proof of Theorem 1.2.}}]
Let $H\in \mathcal{H}_{3,t}^{n,r}$. Suppose $u\in V(H)$ is the vertex with the maximum entry of the Perron vector of $H$ and $w$ is the vertex adjacent to $u$.
Since $H$ is a connected linear $r$-graphs with $n$ vertices, we have $\Delta(H)\leq \frac{n-1}{r-1}$ for any $v\in V(H)$. By Lemmas 3.3 and 4.1, we have
\begin{eqnarray*}
\rho^2(H)&\leq& \frac{1}{r-1}\sum_{v\in N_u(H)}d_v \\
&=& \frac{1}{r-1}\sum_{v\in N_{u}(H)\backslash \{w\}}d_v+\frac{1}{r-1}d_w \\
&\leq& f(n,r,t,d_u,d_w) \\
&\leq& f\bigg(n,r,t,\frac{n-1}{r-1},\frac{n-1}{r-1}\bigg),
\end{eqnarray*}
where the last inequality holds because when $n,r$ and $t$ are fixed, the larger $d_{u}$ and $d_{w}$ are, the larger $f(n,r,t,d_u,d_w)$ is.
\end{proof}

Note that the upper bound in Theorem 1.2 is sharp only if there are two vertices $u,w\in V(H)$ such that $d_{u}(H)=d_{w}(H)=\Delta(H)=\frac{n-1}{r-1}$. Next, we will construct a class of connected Berge-$K_{3,t}$-free linear $r$-graphs with two vertices $u,w$ satisfying $d_{u}=d_{w}=\frac{n-1}{r-1}$.

In \cite{D4}, Khormali and Palmer introduced $[r]^d$ in which $[r]^d$ denotes the integer lattice formed by $d$-tuples from $\{1,2,\cdots,r\}$. The integer lattice $[r]^d$ can be thought of as a hypergraph: the collection of $d$-tuples that are fixed in all but one coordinate forms a hyperedge. Thus, $[r]^d$ is a $d$-regular linear $r$-graph with $r^d$ vertices and $d\cdot r^{d-1}$ hyperedges. In the following we refer to the integer lattice $[r]^d$ as a hypergraph $[r]^d$.
Note that the hyperedges of $[r]^d$ can be partitioned into $d$ classes each of which forms a matching of size $r^{d-1}$. This gives a natural proper hyperedge-coloring of $[r]^d$. In Figure 3 we show two examples of $[r]^d$, where each hyperedge of $[4]^3$ is represented by a line segment.

\begin{figure}[h]
  \centering
  \includegraphics[scale=0.9]{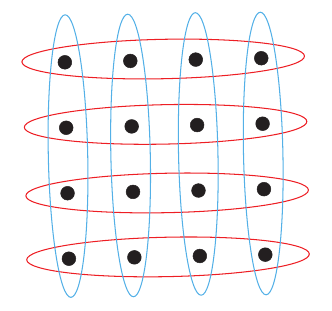}\ \ \
  \includegraphics[scale=1.4]{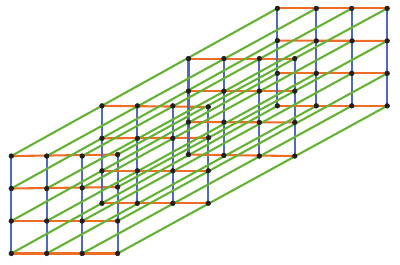}\\
  \caption{\small{Hypergraphs $[4]^2$ and $[4]^3$.}}\label{fig2}
\end{figure}

The union of $k$ hypergraphs $H$ denoted by $kH$. Partition the edges of $(r-1)^{r-2}[r-1]^2$ into $(r-1)^{r-1}$ red edges and $(r-1)^{r-1}$ blue edges. Note that the hypergraph $(r-1)^{r-2}[r-1]^2$ can be obtained from $[r-1]^r$ by deleting $(r-2)(r-1)^{r-1}$ edges of $r-2$ colors.

Let $l=\{u,w,v_1,\cdots,v_{r-2}\}$ be an edge, and $l\vee_{u,w}k(r-1)^{r-2}[r-1]^2$ be an $r$-graph obtained from $l$ and $k(r-1)^{r-2}[r-1]^2$ by inserting $u$ into $k(r-1)^{r-1}$ red edges and inserting $w$ into $k(r-1)^{r-1}$ blue edges (see Figure 4 for an example of $k=1$ and $r=5$).
Let $(r-1)^r\,\big|\,n-r$ and $F:=l\vee_{u,w}\frac{n-r}{(r-1)^2}[r-1]^2$. Then $F$ is a linear $r$-graph with $n$ vertices and $\frac{2(n-r)}{r-1}+1$ edges. Note that $E\big(\frac{n-r}{(r-1)^r}[r-1]^r\big)\big\backslash E\big(\frac{n-r}{(r-1)^2}[r-1]^2\big)$ contains edges of $r-2$ colors, denoted these colors by $c_1,\cdots,c_{r-2}$, where there are $\frac{n-r}{r-1}$ edges for each color.

\begin{figure}[h]
  \centering
  \includegraphics[scale=1.4]{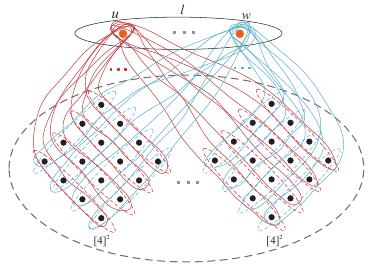}\\
  \caption{\small{The hypergraph $l\vee_{u,w}64[4]^2$.}}\label{fig3}
\end{figure}

Let $\mathcal{G}_{n,r}(F)$ be the set of all linear $r$-graphs with $n$ vertices obtained from $F$ by two steps: for each $1\leq i\leq r-2$, (a) taking any $t-1$ edges from $\frac{n-r}{r-1}$ edges of color $c_i$ of $E\big(\frac{n-r}{(r-1)^r}[r-1]^r\big)\big\backslash E\big(\frac{n-r}{(r-1)^2}[r-1]^2\big)$; (b) inserting $v_i$ to each edge of these $t-1$ edges to form $t-1$ new edges.

For an $r$-graph $H$ and $V'\subseteq V(H)$, let $H-V'$ denote the $r$-graph obtained from $H$ by deleting $V'$ and all their incident edges.
For any $G\in \mathcal{G}_{n,r}(F)$, let $\mathcal{H}_{n,r}(G)$ be the set of all linear $r$-graphs with $n$ vertices obtained from $G$, by embedding a $\frac{n-r}{(r-1)^r}$-partite linear $r$-graph with equal part order $(r-1)^r$ and the maximum degree at most $t-3$ in $G-V(l)$, such that any $H\in \mathcal{H}_{n,r}(G)$ satisfies $d_v(H)\leq t-1$ for any $v\in V(H)\backslash V(l)$.

Let $H\in \mathcal{H}_{n,r}(G)$, we may see that $d_u(H)=d_w(H)=\frac{n-1}{r-1}=\Delta(H)$. Note that any $H\in \mathcal{H}_{n,r}(G)$ satisfies $d_v(H)\leq t-1<t$ for any $v\in V(H)\backslash V(l)$.
By the definition of $\mathcal{H}_{n,r}(G)$, we have $d_{v_i}(H-l)=t-1<t$ for any $1\leq i\leq r-2$ and $H-l$ denote the $r$-graph obtained from $H$ by deleting the edge $l$. Then there are at most two vertices of degree $t$ in a skeleton of ${\rm{Berge}}$-$K_{3,t}$. Thus, $H$ is ${\rm{Berge}}$-$K_{3,t}$-free.

\begin{lem} Suppose that $t>r\geq2$, $(r-1)^r\,\big|\,n-r$ and $H$ is a hypergraph in $\mathcal{H}_{n,r}(G)$. Then $H$ is ${\rm{Berge}}$-$K_{3,t}$-free.
\end{lem}

\begin{proof}[{\bf{Proof of Theorem 1.3.}}]
Let $l=\{u,w,v_1\cdots,v_{r-2}\}$ be an edge and the hypergraph $F$ be defined as above. In the following, we calculate $\rho(F)$. Let ${\bf{x}}$ be the Perron vector of $H$. Then by symmetry, we have $x_u=x_w$, $x_{v_1}=x_{v_2}=\cdots=x_{v_{r-2}}$ and $x_{u_1}=x_{u_2}$ for any $u_1,u_2\in V(F)\backslash V(l)$. We may set $x_u=x_w:=x_1$, $x_{v}:=x_2$ for any $v\in \{v_1,\cdots,v_{r-2}\}$, and $x_{v}:=x_3$ for any $v\in V(F)\backslash V(l)$.
By (1.1), we have
\begin{eqnarray*}
\begin{cases}
\rho(F)x_1^{r-1}=x_1x_2^{r-2}+\frac{n-r}{r-1}x_3^{r-1} \\
\rho(F)x_2^{r-1}=x_1^2x_2^{r-3} \\
\rho(F)x_3^{r-1}=2x_1x_3^{r-2}
\end{cases}
\end{eqnarray*}
and
\begin{eqnarray*}
\rho^r(F)-\rho^{\frac{r}{2}}(F)-\frac{(n-r)2^{r-1}}{r-1}=0.
\end{eqnarray*}

Thus, we obtain $\rho(F)=2^{-\frac{2}{r}}\Big(\sqrt{1+\frac{(n-r)2^{r+1}}{r-1}}+1\Big)^{\frac{2}{r}}$. Since $F\subseteq H\in \mathcal{H}_{n,r}(G)$, by Lemma 4.2, $F$ is ${\rm{Berge}}$-$K_{3,t}$-free. Then we have
$${\rm{spex}}_r^{\rm{lin}}(n,{\text{{\rm{Berge}}-}}K_{3,t})\geq 2^{-\frac{2}{r}}\Big(\sqrt{1+\frac{(n-r)2^{r+1}}{r-1}}+1\Big)^{\frac{2}{r}}.$$
\end{proof}

In 2019, Tait \cite{B15} studied the spectral Tur\'{a}n problem for $K_{s,r}$-minor graphs and gave the following theorem.

\begin{thm}[\cite{B15}] Let $2\leq s\leq t$ and $G$ be a $K_{s,t}$-minor free graph with $n$ vertices, where $n$ is sufficiently large. Then
\begin{eqnarray*}
\rho(G)\leq \frac{1}{2}\big(s+t-3+\sqrt{(s+t-3)^2+4(s-1)(n-s+1)-4(s-2)(t-1)}\big),
\end{eqnarray*}
with equality if and only if $t\ |\ n-s+1$ and $G\cong K_{s-1}\nabla \frac{n-s+1}{t}K_t$, where $\nabla$ represents the join of graphs.
\end{thm}

It is easy to see that $F=K_{2}\nabla \overline{K_{n-2}}$ and $\mathcal{G}_{n,2}(F)=\{F\}$ when $r=2$.
By Theorem 4.3, when $s=3$ and $t\,|\,n-2$, $H^*=K_{2}\nabla \frac{n-2}{t}K_t$ if $\rho(H^*)={\rm{spex}}_2(n,K_{2,t})$. Note that $K_{2}\nabla \frac{n-2}{t}K_t$ can be obtained from $K_{2}\nabla \overline{K_{n-2}}$ by embedding $\frac{n-2}{t}K_t$ in $\overline{K_{n-2}}$.
When $t>r\geq3$, by Lemma 4.2, we know that $H$ is ${\rm{Berge}}$-$K_{3,t}$-free for any $H\in \mathcal{H}_{n,r}(G)$. That is, for any $G\in \mathcal{G}_{n,r}(F)$, we may embed a $\frac{n-r}{(r-1)^r}$-partite linear $r$-graph with equal part order $(r-1)^r$ and the maximum degree at most $t-3$ in $G-V(l)$, and the resulting $r$-graph is still ${\rm{Berge}}$-$K_{3,t}$-free.

Therefore, we speculate that the spectral extremal Berge-$K_{3,t}$-free linear $r$-graphs may be obtained from some hypergraph in $\mathcal{G}_{n,r}(F)$, by embedding a $\frac{n-r}{(r-1)^r}$-partite linear $r$-graph with equal part order $(r-1)^r$ and the maximum degree at most $t-3$.
For further information on the structure of the spectral extremal Berge-$K_{3,t}$-free linear $r$-graphs, we propose the following conjecture.

\begin{con}
Suppose $t>r\geq2$ and $(r-1)^r\,|\,n-r$. If $\rho(H^*)={\rm{spex}}_r^{\rm{lin}}(n,{\text{{\rm{Berge}}-}}K_{3,t})$, then $H^*\in \mathcal{H}_{n,r}(\widehat{G})$, where $\widehat{G}$ is the $r$-graph with the maximum spectral radius in $\mathcal{G}_{n,r}(F)$.
\end{con}

\noindent{\bf{Declaration of interest}}

The authors declare no known conflicts of interest.

\noindent{\bf{Acknowledgements}}

The authors would like to express their sincere gratitude to the editor and the referees for a very careful reading of the paper and for all their insightful comments and valuable suggestions, which led to a number of improvements in this paper.

\end{document}